\documentclass[10pt,a4paper]{article}
\usepackage{latexsym,amsthm,amssymb,amscd,amsmath}
\newcounter{alphthm}
\theoremstyle{plain}
\newtheorem{theorem}{Theorem}[section]
\newtheorem{lemma}[theorem]{Lemma}
\newtheorem{proposition}[theorem]{Proposition}
\newtheorem{cor}[theorem]{Corollary}
\theoremstyle{definition}
\newtheorem{defn}[theorem]{Definition}
\newtheorem{rem}[theorem]{Remark}

\newcommand{\be}{\begin{equation}}
\newcommand{\ee}{\end{equation}}
\newcommand{\ben}{\begin{enumerate}}
\newcommand{\een}{\end{enumerate}}

\begin{document}
\title{Almost Paracontact Finsler Structures\\ on Vector Bundle}
\author{E. Peyghan, A. Tayebi and E. Sharahi}
\maketitle

\maketitle
\begin{abstract}
In this paper,  we define almost paracontact and normal almost paracontact Finsler structures on a vector bundle  and find some conditions for integrability of these structures. We define paracontact metric, para-Sasakian  and K-paracontact Finsler structures and study some properties of these structures.  For a K-paracontact Finsler structure, we find the  vertical and horizontal flag curvatures. Then, we define vertical $\phi$-flag curvature and prove that every locally symmetric para-Sasakian Finsler structure has negative vertical $\phi$-flag curvature. Finally, we define the horizontal and vertical Ricci tensors of a  para-Sasakian Finsler structure and study some curvature properties of  them.
\end{abstract}

\section{Introduction}
  The contact geometry has a very deep relation with physical concepts. At first
this geometry introduced by Sophus Lie in his works on PDEs. Contact theory
is in contrast with foliation theory. In the foliation theory, when one study a
distribution, it is very important for the distribution to be integrable. But on the
other hand in contact theory, we are interested in study on a distribution so that
it does nowhere integrable (even locally). Although this does not occur for any
one-dimensional distribution, but in upper-dimensional distributions we can
find such structures that their vector fields are not tangent to any submanifold of
the main manifold.

The notion of paracontact structure is alongside with contact
structure. This kind of structures first introduced by S\`{a}to \cite{S}. Then Sasaki studied some  interesting concepts of these structures  as he
had done for contact  structures \cite{Sa}\cite{Sa1}. Recently, many mathematician such as  Bejan, Calvaruso, Dru\c{t}\v{a}, Ivanov, Kaneyuki,  Montano and Zamakovo studied interesting properties of these structures \cite{BD, C, IVZ, KW, M, MEM, Z, Z1}.

The notion of vector bundle is one of  important geometric objects that have interesting applications in physics \cite{VV, V}. The importance of vector bundles caused to definition of similar structures with almost contact (paracontact) structures on vector bundle by  Sinha, Prasad and Yadav \cite{SP, SY}. In \cite{SP}, they introduced an almost paracontact Finsler structure $(\phi, \eta, \xi)$ on a vector bundle $E$ and consider the  condition  $G(\phi X, \phi Y)=G(X, Y)-\eta(X)\eta(Y)$ to compatibility of this structure with respect to the metric $G$ on $E$. It is remarkable that, an almost paracontact structure on a manifold $N$ is a set $(\phi,\xi,\eta)$ where $\phi$ is a
tensor field of type (1,1), $\xi$ a vector field and $\eta$ an 1-form such that $\eta(\xi)=1$, $\phi(\xi)=0$, $\eta\circ\phi=0$ and $\phi^2=I-\eta\otimes\xi$, where $I$ denotes the Kronecker tensor field.

In this paper,  we define almost paracontact Finsler structures and normal almost paracontact Finsler structures on vector bundle $E$ and introduce some conditions for integrability (normality) of these structures. We give some equivalent conditions for normality of an almost paracontact Finsler structure.

Then by using a pseudo-metric $G$ on $E$, similar to \cite{Z}, we consider the following condition to compatibility of this structure
\[
G(\phi X, \phi Y)=-G(X, Y)+\eta(X)\eta(Y).
\]
We define paracontact metric Finsler structure, para-Sasakian Finsler structure and K-paracontact Finsler structure. We find some conditions under which a Paracontact metric Finsler structure is a K-paracontact structure. Then we get  some conditions under which a paracontact metric Finsler structure on vector bundle $E$ reduces to a K-paracontact Finsler
structure. For a K-paracontact Finsler structure  on vector bundle $E$, we find the the vertical and horizontal flag curvatures. We define vertical $\phi$-flag curvature and prove that every locally symmetric para-Sasakian Finsler structure has vertical $\phi$-flag curvature $-\frac{1}{4}$.

Finally, we define the horizontal and vertical Ricci tensors of a  para-Sasakian
Finsler manifold and study some curvature properties of  them.

%------------------------------------------------------------------------------------------------------------------------------------------------------------------------------
\section{Preliminaries}
Let $E(M)=(E, \pi, M)$ be a vector bundle with the $(n+m)$-dimensional total space $E$, $n$-dimensional base space $M$ and the projection map $\pi$, such that $\pi: E\rightarrow M$, $u\in E\rightarrow\pi(u)=x\in M$ where $u=(x, y)$ and $y=\pi^{-1}(x)$ is the fibre of $E(M)$ over $x$. We denote by $V_uE$ the local fibre of the vertical bundle $VE$ at $u\in E$ and by $H_uE$ the complementary space of $V_uE$ in the tangent space $T_uE$ at $u$ to the total space $E$. Thus we have
\begin{equation}\label{decom}
T_uE=H_uE\oplus V_uE.
\end{equation}
A nonlinear connection $N$ on the total space $E$ of $E(M)$ is a differentiable distribution $H: E\rightarrow T_uE$, $u\in E\rightarrow H_u\subset T_uE$ with the property (\ref{decom}) (see \cite{MHSS}).

We denote by $(x^i, y^a)$, $i=1,\ldots, n$, $a=1,\ldots, m$,  the canonical coordinates of a point $u\in E$. Then $\{\frac{\partial}{\partial x^i}, \frac{\partial}{\partial y^a}\}$ is the natural basis and $\{dx^i, dy^a\}$ is it's dual basis on $E$. It is easy to see that $\{\frac{\delta}{\delta x^i}, \frac{\partial}{\partial y^a}\}$ is the basis on $E$ adapted to decomposition (\ref{decom}) and $\{dx^i, \delta y^a\}$ is the dual basis (cobasis) of it, where
\begin{equation}\label{adapt}
\frac{\delta}{\delta x^i}=\frac{\partial}{\partial x^i}-N^a_i\frac{\partial}{\partial y^a},\ \ \ \delta y^a=dy^a+N^a_idx^i,
\end{equation}
and $N^a_i$ are the coefficients of a nonlinear connection $N$. Now, we consider the horizontal and the vertical projectors $h$ and $v$ of the
nonlinear connection, which are determined by the direct decomposition (\ref{decom}).
These projectors can be expressed with respect to the adapted basis as follows:
\begin{equation}
h=\frac{\delta}{\delta x^i}\otimes dx^i,\ \ \ v=\frac{\partial}{\partial y^a}\otimes\delta y^a.
\end{equation}
Using the above projectors, any vector field $X$ on $E$ can be uniquely written as $X=hX + vX$. In the following, we adopt the notations
\[
hX=X^H,\ \ \ vX=X^V
\]
and we say $X^H$ and $X^V$ are horizontal and vertical components of $X$. Thus, any
vector field $X$ on $E$ can be uniquely written in the form
\begin{equation}
X=X^H+X^V.
\end{equation}
In the adapted basis, we have $X=X^i(x, y)\frac{\delta}{\delta x^i}+\bar{X}^a(x, y)\frac{\partial}{\partial y^a}$ and
\begin{equation}\label{d1}
X^H=X^i(x, y)\frac{\delta}{\delta x^i},\ \ \ \ X^V=\bar{X}^a(x, y)\frac{\partial}{\partial y^a}.
\end{equation}
Now, let $\omega$ be a 1-form on $E$. Then it can be uniquely written as $\omega=\omega^H+\omega^V$. In the adapted basis,  we have $\omega=\omega_i(x, y)dx^i+\bar{\omega}_a(x, y)\delta y^a$ and
\begin{equation}\label{d2}
\omega^H=\omega_i(x, y)dx^i,\ \ \ \omega^V=\bar{\omega}_a(x, y)\delta y^a.
\end{equation}
A tensor field $T$ on the vector bundle $E$ is called \textit{distinguished tensor field (briefly, a d-tensor)} of type $\left(
\begin{array}{cc}
p&r\\
q&s
\end{array}
\right)$
if it has the following property
\begin{align*}
T&(\omega_{i_1},\ldots,\omega_{i_p}, \omega_{a_1},\ldots,\omega_{a_r},X_{j_1},\ldots, X_{j_q}, X_{b_1},\ldots, X_{b_s})\\
&=T({\omega^H_{i_1}},\ldots,{\omega^H_{i_p}}, {\omega^V_{a_1}},\ldots,{\omega^V_{a_r}},{X^H_{j_1}},\ldots, {X^H_{j_q}}, {X^V_{b_1}},\ldots, {X^V_{b_s}}),
\end{align*}
where $\omega_{i_k}$, $\omega_{a_l}$, ($k=1,\ldots,p$, $l=1,\ldots,r$) are 1-forms on $E$ and $X_{j_v}$, $X_{b_w}$, ($v=1,\ldots, q$, $w=1,\ldots, s$) are vector fields on $E$. For instance, the components $X^H$ and $X^V$ from (\ref{d1}) of a vector field $X$ are d-vector fields. Also the components $\omega^H$ and $\omega^V$ of an 1-form $\omega$, from (\ref{d2}) are
d-1-form fields. In the adapted basis $\{\frac{\delta}{\delta x^i}, \frac{\partial}{\partial y^a}\}$ and adapted cobasis $\{dx^i, \delta y^a\}$, $T$ is expressed by
\begin{align*}
T&=T^{i_1,\ldots,i_p,a_1,\ldots,a_r}_{j_1,\ldots,j_q,b_1,\ldots,b_s}\frac{\delta}{\delta x^{i_1}}\otimes\ldots\otimes\frac{\delta}{\delta x^{i_p}}\otimes\frac{\partial}{\partial y^{a_1}}\otimes\ldots\otimes\frac{\partial}{\partial y^{a_r}}\\
&\ \ \ \otimes dx^{j_1}\otimes\ldots\otimes dx^{j_q}\otimes\delta y^{b_1}\otimes\ldots\otimes\delta y^{b_s}.
\end{align*}

\bigskip

A linear connection $D$ on $E$ is called a \textit{distinguished connection (briefly, d-connection)} if it preserves by parallelism the horizontal distribution, that is $D h=0$. Since $Id=h+v$, then  $D h=0$ implies that $D v=0$. Thus a d-connection preserves by the parallelism the vertical distribution. Therefore, we can write
\begin{align*}
D_XY&=(D_XY^H)^H+(D_XY^V)^V,\\
D_X\omega&=(D_X\omega^H)^H+(D_X\omega^V)^V,
\end{align*}
where $X, Y$ are vector fields on $E$ and $\omega$ is a 1-form on $E$.

A d-connection, with respect to the adapted basis, has the following form
\begin{equation}
\left\{
\begin{array}{cc}
D_{\frac{\delta}{\delta x^i}}\frac{\delta}{\delta x^j}=F^k_{ij}\frac{\delta}{\delta x^k},&D_{\frac{\delta}{\delta x^i}}\frac{\partial}{\partial y^b}=\bar{F}^c_{ib}\frac{\partial}{\partial y^c},\\
D_{\frac{\partial}{\partial y^a}}\frac{\delta}{\delta x^j}=C^k_{aj}\frac{\delta}{\delta x^k},&D_{\frac{\partial}{\partial y^a}}\frac{\partial}{\partial y^b}=\bar{C}^c_{ab}\frac{\partial}{\partial y^c}.
\end{array}
\right.
\end{equation}
For this connection, there is an associated pair of operators in the algebra of d-tensor fields. For any vector field $X$ on $E$, set
\[
D^H_XY=D_{X^H}Y,\ \ D^V_XY=D_{X^V}Y \ \ D^H_Xf=X^H(f),\ \ D^V_Xf=X^V(f),
\]
where $Y$ is a vector field and $f$ is a smooth function on $E$. We call $D^H$ ($D^V$) the operator of $h$-covariant ($v$-covariant) derivation. If $\omega$ is a 1-form on $E$, we define
\begin{align*}
(D^H_X\omega)Y&=X^H(\omega(Y))-\omega(D^H_XY),\\
(D^V_X\omega)Y&=X^V(\omega(Y))-\omega(D^V_XY),
\end{align*}
for any vector fields $X, Y$ on $E$.

Now, we consider the pseudo-metric structure $G$ on $E$ which is symmetric and non-degenerate, as $G=G^H+G^V$,  where $G^H(X, Y)=G(X^H, Y^H)$ is of type
$\left(
\begin{array}{cc}
0&0\\
2&0
\end{array}
\right)
$, symmetric and non-degenerate on $H_uE$ and $G^V(X, Y)=G(X^V, Y^V)$ is of type $
\left(
\begin{array}{cc}
0&0\\
0&2
\end{array}
\right)
$, symmetric and non-degenerate on $V_uE$. In the adapted basis,  we can write
\[
G=g_{ij}(x, y)dx^i\otimes dx^j+h_{ab}(x, y)\delta y^i\otimes\delta y^j.
\]

A d-connection $D$ on $E$ is called a \textit{metrical d-connection} with respect to $G$ if $D_XG=0$ holds for every vector field $X$ on $E$.

For a d-connection $D$, we consider the torsion $T$ defined by
\[
T(X, Y)=D_XY-D_YX-[X, Y],\ \ \ \forall X, Y\in\chi(E),
\]
where $\chi(E)$ is the set of all vector fields on $E$. The torsion of a d-connection $D$ on $E$ is completely
determined by the following five tensor fields:
\begin{align*}
T^H(X^H, Y^H)&=D^H_XY^H-D^H_YX^H-[X^H, Y^H]^H,\\
T^V(X^H, Y^H)&=-[X^H, Y^H]^V,\\
T^H(X^H, Y^V)&=-D^V_YX^H-[X^H, Y^V]^H,\\
T^V(X^H, Y^V)&=D^H_XY^V-[X^H, Y^V]^V,\\
T^V(X^V, Y^V)&=D^V_XY^V-D^V_YX^V-[X^V, Y^V]^V,
\end{align*}
which are called $(h)h$-torsion, $(v)h$-torsion, $(h)hv$-torsion, $(v)hv$-torsion and $(v)v$-torsion, respectively.
A d-connection $D$ is said to be \textit{symmetric} if the $(h)h$-torsion and $(v)v$-torsion vanish. In this paper, we use symmetric metrical d-connection and we call it \textit{Finsler connection}.
It is easy to see that the following relations hold for a Finsler connection:
\begin{eqnarray}
\nonumber2G(D_X^HY^H\!\!\!\!\!\!\!\!&,&\!\!\!\!\!\!\!\! \ Z^H)=X^HG(Y^H, Z^H)+Y^HG(X^H, Z^H)-Z^HG(X^H,Y^H)\\
\!\!\!\!&+&\!\!\!\!G([X^H, Y^H], Z^H)-G([X^H, Z^H], Y^H)-G([Y^H, Z^H], X^H),\label{l2}\\
\nonumber2G(D_X^VY^V\!\!\!\!\!\!\!\!&,&\!\!\!\!\!\!\!\! \ Z^V)=X^VG(Y^V, Z^V)+Y^VG(X^V, Z^V)-Z^VG(X^V,Y^V)\\
\!\!\!\!&+&\!\!\!\! G([X^V, Y^V], Z^V)-G([X^V, Z^V], Y^V)-G([Y^V, Z^V], X^V)\label{l20}.
\end{eqnarray}
Finally, we consider the curvature of a Finsler connection $D$ as following
\[
R(X, Y)Z=D_XD_YZ-D_YD_XZ-D_{[X, Y]}Z,\ \ \ \forall X, Y, Z\in\chi(E).
\]
As $D$ preserves by parallelism, on the horizontal and the vertical distributions,
from the above equation, we see that the operator $R(X, Y)$ carries horizontal vector fields
into horizontal vector fields and vertical vector fields into verticals. Consequently,
we have the following
\[
R(X, Y)Z=(R(X, Y)Z^H)^H+(R(X, Y)Z^V)^V\ \ \ \forall X, Y, Z\in\chi(E).
\]
Since $R(X, Y)$ is skew symmetric with
respect to $X$ and $Y$, then the curvature of a Finsler connection $D$ on $E$ is completely determined by the following six tensor fields
\begin{equation}\label{Rie}
\left\{
\begin{array}{cc}
R(X^H, Y^H)Z^H=D^H_XD^H_YZ^H-D^H_YD^H_XZ^H-D_{[X^H, Y^H]}Z^H,&\\
R(X^H, Y^H)Z^V=D^H_XD^H_YZ^V-D^H_YD^H_XZ^V-D_{[X^H, Y^H]}Z^V,&\\
R(X^V, Y^H)Z^H=D^V_XD^H_YZ^H-D^H_YD^V_XZ^H-D_{[X^V, Y^H]}Z^H,&\\
R(X^V, Y^H)Z^V=D^V_XD^H_YZ^V-D^H_YD^V_XZ^V-D_{[X^V, Y^H]}Z^V,&\\
R(X^V, Y^V)Z^H=D^V_XD^V_YZ^H-D^V_YD^V_XZ^H-D_{[X^V, Y^V]}Z^H,&\\
R(X^V, Y^V)Z^V=D^V_XD^V_YZ^V-D^V_YD^V_XZ^V-D_{[X^V, Y^V]}Z^V.
\end{array}
\right.
\end{equation}
We call the first and the sixth equation of (\ref{Rie}) as \textit{horizontal curvature} and \textit{vertical curvature} of $D$.
%--------------------------------------------------------------------------------------------
%----------------------------------------------------------------------------------------------------------
\section{Almost Praracontact Finsler Structure}
We consider tensor field $\phi$, 1-form $\eta$ and vector field $\xi$ on $E$, as follows
\begin{align}
\phi&=\phi^i_j(x, y)\frac{\delta}{\delta x^i}\otimes dx^j+\bar{\phi}^a_b(x, y)\frac{\partial}{\partial y^a}\otimes\delta y^a,\label{pa1}\\
\eta&=\eta_i(x, y)dx^i+\bar{\eta}_a(x, y)\delta y^a,\ \ \ \xi=\xi^i(x, y)\frac{\delta}{\delta x^i}+\bar{\xi}^a(x, y)\frac{\partial}{\partial y^a}\label{pa2}.
\end{align}
\begin{defn}
\emph{Suppose that $\phi$, $\eta$ and $\xi$ are given by (\ref{pa1}) and (\ref{pa2}) on $E$ such that
\begin{equation}
\phi^2=I-\eta^H\otimes\xi^H-\eta^V\otimes\xi^V,\ \ \eta^H(\xi^H)=\eta^V(\xi^V)=1,\label{C1}
\end{equation}
where
\[
\eta^H=\eta_i(x, y)dx^i, \ \ \eta^V=\bar{\eta}_a(x, y)\delta y^a,\ \ \xi^H=\xi^i(x, y)\frac{\delta}{\delta x^i},\ \ \xi^V=\bar{\xi}^a(x, y)\frac{\partial}{\partial y^a}.
\]
Then $(\phi, \eta, \xi)$ is called an almost paracontact Finsler structure on $E$ and $E$ is called an almost paracontact Finsler vector bundle.}
\end{defn}

\bigskip

Now, we are going to consider some properties of an almost paracontact Finsler structure. First, we prove the following.

\begin{theorem}\label{feri1}
Suppose that $E$ has an almost paracontact Finsler structure, then the following hold
\begin{equation}
\phi(\xi^H)=\phi(\xi^V)=0, \ \ \eta^H\circ\phi=\eta^V\circ\phi=0.
\end{equation}
\end{theorem}
\begin{proof}
By (\ref{C1}) and  $\eta^V(\xi^H)=0$, we have
\[
\phi^2(\xi^H)=\xi^H-\eta^H(\xi^H)\xi^H=0.
\]
Then $\phi(\xi^H)=0$ or $\phi(\xi^H)$ is a nontrivial eigenvector of $\phi$ corresponding to the eigenvalue 0. Since $\phi(\xi^H)\in HE$, then  $\eta^V(\phi(\xi^H))=0$. Using (\ref{C1}),  we obtain
\[
0=\phi^2(\phi(\xi^H))=\phi(\xi^H)-\eta^H(\phi(\xi^H))\xi^H \  \ \textrm{ or} \ \ \ \phi(\xi^H)=\eta^H(\phi(\xi^H))\xi^H.
\]
Now, if $\phi(\xi^H)$ is nontrivial eigenvector of the eigenvalue 0, then $\eta^H(\phi(\xi^H))\neq0$. Thus we have
\[
0=\phi^2(\xi^H)=\eta^H(\phi(\xi^H))\phi(\xi^H)=(\eta^H(\phi(\xi^H)))^2\xi^H\neq 0,
\]
which is a contradiction. Therefore  $\phi(\xi^H)=0$. Similarly we get $\phi(\xi^V)=0$.

On the other hand, since $\phi(\xi^H)=0$, then we get
\begin{align*}
\eta^H(\phi(X))\xi^H&=\eta^H(\phi(X^H))\xi^H=\phi(X^H)-\phi^3(X^H)\\
&=\phi(X^H)-\phi(X^H)+\phi(\eta^H(X^H)\xi^H)=0,
\end{align*}
for any $X\in \chi(E)$. Hence $\eta^H\circ\phi=0$. Similarly we have $\eta^V\circ\phi=0$.
\end{proof}

\bigskip

\begin{rem}
Let us put
\[
\phi^H=\phi^i_j(x, y)\frac{\delta}{\delta x^i}\otimes dx^j\ \ and \ \ \ \phi^V=\bar{\phi}^a_b\frac{\partial}{\partial y^a}\otimes\delta y^b,
\]
Then by Theorem \ref{feri1}, we deduce that $(\phi^H, \eta^H, \xi^H)$ and $(\phi^V, \eta^V, \xi^V)$ are almost paracontact structures on subbundles $HE$ and $VE$, respectively. Therefore for all $u\in E$,  $H_uE$ and $V_uE$ have odd dimensions $n=2k_1+1$ and $m=2k_2+1$, respectively. Therefore we have $dim E=m+n=2(k_1+k_2)+2$, i.e., the dimension of an almost paracontact Finsler vector bundle should be even.
\end{rem}

\bigskip

\begin{proposition}
Let $E^{2(k_1+k_2)+2}$ has an almost paracontact Finsler structure $(\phi, \eta, \xi)$. Then the endomorphism $\phi$ has rank $2(k_1+k_2)$.
\end{proposition}
\begin{proof}
It is sufficient to show that $\ker\phi=<\xi^H>\oplus<\xi^V>$. Since $\phi\xi^H=\phi\xi^V=0$, then we have $<\xi^H>\oplus<\xi^V>\subseteq\ker\phi$. Now, let $\bar{\xi}\in\ker\phi$. Then $\phi\bar{\xi}=0$ and (\ref{C1}) give us
\[
\bar{\xi}=\eta^H(\bar{\xi})\xi^H+\eta^V(\bar{\xi})\eta^V\in<\xi^H>\oplus<\xi^V>,
\]
i.e., $\ker\phi\subseteq<\xi^H>\oplus<\xi^V>$. Thus $\ker\phi=<\xi^H>\oplus<\xi^V>$.
\end{proof}

\bigskip

We say that an almost paracontact Finsler structure $(\phi, \eta, \xi)$ on the vector bundle $E$ is \textit{normal}, if the following holds
\begin{equation}
N^{(1)}(X, Y)=N_\phi(X, Y)-d\eta^H(X, Y)\xi^H-d\eta^V(X, Y)\xi^V=0,
\end{equation}
where $X, Y$ are vector fields  on $E$.

Now, we are going to give some equivalent conditions for normality of $(\phi, \eta, \xi)$. For this reason, we introduce three tensors $N^{(2)}$, $N^{(3)}$ and $N^{(4)}$ and  show that the vanishing of $N^{(1)}$ implies the vanishing of these tensors. First, we define the tensor $N^{(2)}$ on $T_uE$ as follows
\begin{eqnarray}
N^{(2)}(X^H, Y^H)\!\!\!\!&=&\!\!\!\!(\pounds^H_{\phi X}\eta^H)(Y^H)-(\pounds^H_{\phi Y}\eta^H)(X^H),\\
N^{(2)}(X^V, Y^V)\!\!\!\!&=&\!\!\!\!(\pounds^V_{\phi X}\eta^V)(Y^V)-(\pounds^V_{\phi Y}\eta^V)(X^V),\\
N^{(2)}(X^V, Y^H)\!\!\!\!&=&\!\!\!\!(\pounds^V_{\phi X}\eta^H)(Y^H)+(\pounds^V_{\phi X}\eta^V)(Y^H)\nonumber\\
&&\!\!\!\!-(\pounds^H_{\phi Y}\eta^H)(X^V)-(\pounds^H_{\phi Y}\eta^V)(X^V).
\end{eqnarray}
To define $N^{(3)}$ and $N^{(4)}$, we consider the following cases:\\\\
\noindent
\textbf{Case 1:} For $X^H, \xi^H\in H_uE$, we define
\begin{equation}
N^{(3)}(X^H)=(\pounds^H_\xi\phi)(X^H),\ \ \ \ N^{(4)}(X^H)=(\pounds^H_\xi\eta^H)(X^H).
\end{equation}
\noindent
\textbf{Case 2:} For $X^V, \xi^V\in V_uE$, we define
\begin{equation}
N^{(3)}(X^V)=(\pounds^V_\xi\phi)(X^V),\ \ \ \ N^{(4)}(X^V)=(\pounds^V_\xi\eta^V)(X^V).
\end{equation}
\noindent
\textbf{Case 3:} For $X^H\in H_uE$ and $\xi^V\in V_uE$, we define
\begin{equation}
N^{(3)}(X^H)=(\pounds^V_\xi\phi)(X^H),\ \ \ \ N^{(4)}(X^H)=(\pounds^V_\xi\eta^H)(X^H).
\end{equation}
\noindent
\textbf{Case 4:} For $X^V\in V_uE$ and $\xi^H\in H_uE$, we define
\begin{equation}
N^{(3)}(X^V)=(\pounds^H_\xi\phi)(X^V),\ \ \ \ N^{(4)}(X^V)=(\pounds^H_\xi\eta^V)(X^V).
\end{equation}

\bigskip

\begin{theorem}
For any almost paracontact Finsler structure $(\phi, \eta, \xi)$ the vanishing of $N^{(1)}$ implies the vanishing of $N^{(2)}$, $N^{(3)}$ and $N^{(4)}$.
\end{theorem}
\begin{proof}
If $N^{(1)}=0$, then for $ X^H $ and $ \xi^H $ we have
\begin{eqnarray}\label{salar}
\nonumber0\!\!\!\!\!&=&\!\!\!\!\!N^{(1)}(X^H,\xi^H)\\
\nonumber\!\!\!\!\!&=&\!\!\!\!\! \phi^2[X^H,\xi^H]+[\phi X^H,\phi \xi^H]-\phi[\phi X^H,\xi^H]-\phi[X^H,\phi \xi^H]\nonumber\\
&&\!\!\!-d\eta^H(X^H,\xi^H)\xi^H-d\eta^V(X^H, \xi^H)\xi^V\nonumber\\
\!\!\!\!\!&=&\!\!\!\!\!\phi^2[X^H,\xi^H]-\phi[\phi X^H,\xi^H]-d\eta^H(X^H,\xi^H)\xi^H-d\eta^V(X^H,\xi^H)\xi^V.
\end{eqnarray}
Applying $\eta^H $ to (\ref{salar}) implies
\[
d\eta^H(X^H,\xi^H)=0,
\]
which gives
\begin{eqnarray*}
N^{(4)}(X^H)=(\pounds^H_\xi\eta^H)(X^H)\!\!\!\!\!&=&\!\!\!\!\! \xi^H\big(\eta^H(X^H)\big)-\eta^H[\xi^H, X^H]\\
\!\!\!\!\!&=&\!\!\!\!\!-d\eta^H(X^H, \xi^H)=0.
\end{eqnarray*}
Since $d\eta^H(X^H,\xi^H)=0$, then by (\ref{salar}) we have
\begin{equation}\label{salar20}
0=\phi^2[X^H,\xi^H]-\phi[\phi X^H,\xi^H]=\phi\big((\pounds^H_\xi\phi)X^H\big).
\end{equation}
Similar to (\ref{salar}), we obtain
\begin{eqnarray*}
0\!\!\!\!\!&=&\!\!\!\!\!\eta^H\Big(N^{(1)}(\phi X^H, \xi^H)\Big)=d\eta^H(\xi^H, \phi X^H),\\
0\!\!\!\!\!&=&\!\!\!\!\!\eta^V\Big(N^{(1)}(\phi X^H, \xi^H)\Big)=d\eta^V(\xi^H, \phi X^H).
\end{eqnarray*}
which imply that
\[
\eta^H([\xi^H, \phi X^H])=0,\ \ \ \eta^V([\xi^H, \phi X^H])=0.
\]
Applying $\phi$ to (\ref{salar20}) and using the above equation, we have $(\pounds^H_\xi\phi)X^H=0$, i.e., $N^{(3)}(X^H)=0$. Applying $\eta^H$ to the following
\begin{eqnarray*}
\nonumber0\!\!\!\!\!&=&\!\!\!\!\!N^{(1)}(\phi X^H,Y^H)=[X^H,\phi Y^H]-\eta^H(X^H)[\xi^H,\phi Y^H]+\phi Y^H\big(\eta^H(X^H)\big)\xi^H\\
\nonumber&&-\phi[X^H,Y^H]-\phi[\phi X^H,\phi Y^H]+[\phi X^H,Y^H]-\phi X^H\big(\eta^H(Y^H)\big)\xi^H\\
&&+\eta^H(X^H)\phi[\xi^H, Y^H]+\eta^V[\phi X^H, Y^H]\xi^V.\label{C2}
\end{eqnarray*}
and using $\eta^H([\xi^H, \phi X^H])=0$, we get
\begin{align*}
0&=-\eta^H[\phi Y^H,X^H]+\phi Y^H(\eta^H(X^H))+\eta^H[\phi X^H,Y^H]-\phi X^H(\eta^H(Y^H))\\
&=(\pounds_{\phi Y}^{H}\eta^H)X^H-(\pounds_{\phi X}^{H}\eta^H)Y^H.
\end{align*}
Thus $ N^{(2)}(X^H,Y^H)=0$. By the similar way, we can conclude the vanishing of $N^{(2)}$, $N^{(3)}$ and $N^{(4)}$ from the vanishing of $N^{(1)}$, when $X^V$ and $Y^V$ belong to $V_uE$. Now we prove the result when one of them belongs to $V_uE$ and another belongs to $H_uE$.

Similar to (\ref{salar}), the vanishing of $N^{(1)}$ implies that
\begin{eqnarray}\label{Nij2}
\nonumber 0\!\!\!\!\!&=&\!\!\!\!\! N^{(1)}(X^V,\xi^H)\\
\!\!\!\!\!&=&\!\!\!\!\!\phi^2[X^V,\xi^H]-\phi[\phi X^V,\xi^H]-d\eta^H(X^V,\xi^H)\xi^H-d\eta^V(X^V,\xi^H)\xi^V.
\end{eqnarray}
By applying  $\eta^V$ and $\eta^H$ to (\ref{Nij2}),  we get
\begin{equation}\label{salar100}
d\eta^V(X^V,\xi^H)=0,\ \ \ \ d\eta^H(X^V,\xi^H)=0.
\end{equation}
But we have
\[
N^{(4)}(X^V)=(\pounds_{\xi}^{H}\eta^V)(X^V)=\xi^H(\eta^V(X^V))-\eta^V[\xi^H,X^V]=-d\eta^V(X^V,\xi^H).
\]
Therefore the first part of (\ref{salar100}) gives us $ N^{(4)}(X^V)=0 $. Using (\ref{Nij2}) and (\ref{salar100}), we obtain
\begin{eqnarray*}
0= \phi(N^{(1)}(X^V,\xi^H))\!\!\!\!\!&=&\!\!\!\!\!\phi[X^V,\xi^H]-[\phi X^V,\xi^H]\\
\!\!\!\!\!&=&\!\!\!\!\!(\pounds^H_\xi\phi)(X^V)\\
\!\!\!\!\!&=&\!\!\!\!\!N^{(3)}(X^V).
\end{eqnarray*}
Therefore $N^{(3)}(X^V)=0$. By similar way to (\ref{Nij2}), we obtain
\begin{eqnarray}
0\!\!\!\!&=&\!\!\!\!\eta^H\Big(N^{(1)}(\xi^V, \phi Y^H)\Big)=-d\eta^H(\xi^V, \phi Y^H),\label{o1}\\
0\!\!\!\!&=&\!\!\!\!\eta^V\Big(N^{(1)}(\xi^V, \phi Y^H)\Big)=-d\eta^V(\xi^V, \phi Y^H)\label{o2}
\end{eqnarray}
which give us
\be
\eta^H[\xi^V, \phi Y^H]=0,\ \ \ \eta^V[\xi^V, \phi Y^H]=0.\label{o3}
\ee
Using  (\ref{o3}) and the above relations we get
\begin{align*}
0&=\eta\Big(N^{(1)}(\phi X^V,Y^H)\Big)\\
&=\eta^H([X^V, \phi Y^H])+\eta^V([X^V,\phi Y^H])+\phi Y^H(\eta^V(X^V))+\eta^V([\phi X^V,Y^H])\\& \ \ \  \ \ \ -\phi X^V(\eta^H(Y^H))+\eta^H([\phi X^V, Y^H])\\
&=-N^{(2)}(X^V, Y^H),
\end{align*}
i.e., $N^{(2)}(X^V,Y^H)=0$.
\end{proof}
%-------------------------------------------------------------
%--------------------------------------------------------------
\section{Paracontact Finsler Structures}
A pseudo-metric structure $G$ on $E$ satisfying the conditions
\begin{align}
G^H(\phi X, \phi Y)&=-G^H(X, Y)+\eta^H(X)\eta^H(Y),\label{com}\\
G^V(\phi X, \phi Y)&=-G^V(X, Y)+\eta^V(X)\eta^V(Y).\label{com1}
\end{align}
is said to be compatible with the structure $(\phi, \eta, \xi)$. In this case, the quadruplet $(\phi, \eta, \xi, G)$ is called an almost paracontact metric Finsler structure and $E$ is called an almost paracontact metric Finsler vector bundle. From (\ref{com}) and (\ref{com1}) we deduce
\begin{equation}
G(\phi X, \phi Y)=-G(X, Y)+\eta^H(X)\eta^H(Y)+\eta^V(X)\eta^V(Y).
\end{equation}
By (\ref{com}) and (\ref{com1}) we have
\begin{equation}\label{com20}
G^H(X, \xi)=\eta^H(X),\ \ \ G^V(X, \xi)=\eta^V(X),
\end{equation}
which give us $G(X, \xi)=\eta(X)$. Using (\ref{com})-(\ref{com20}), one can also obtains
\begin{equation}
G(X^H, \phi Y^H)=-G(\phi X^H, Y^H),\ \ G(X^V, \phi Y^V)=-G(\phi X^V, Y^V).
\end{equation}
Now, we define the fundamental 2-form $\Phi$ by
\begin{equation}
\Phi(X, Y)=G(X, \phi Y),\ \ \ \forall X, Y\in\chi(E),
\end{equation}
which gives
\begin{align}
\Phi(X^H, Y^H)&=G^H(X, \phi Y),\ \ \ \ \Phi(X^V, Y^V)=G^V(X, \phi Y),\label{form}\\
\Phi(X^V, Y^H)&=-\Phi(Y^H, X^V)=G(X^V, \phi Y^H)=0.
\end{align}

\begin{defn}
\emph{Almost paracontact metric Finsler structure $(\phi, \eta, \xi, G)$ is called paracontact metric Finsler structure if}
\begin{equation}
d\eta^H(X, Y)=\Phi(X^H, Y^H), \ \ \ \ d\eta^V(X, Y)=\Phi(X^V, Y^V).\label{com2}
\end{equation}
\end{defn}

\smallskip

By (\ref{form}) and (\ref{com2}), it follows that $d\eta(X, Y)=G(X, \phi Y)$. Then we get the following
\[
d\eta(X^H, Y^H)=G(X^H, \phi Y^H)=G^H(X, \phi Y)=d\eta^H(X, Y).
\]
Similarly we obtain
\[
d\eta(X^V, Y^V)=d\eta^V(X, Y)\ \ \ \textrm{and}  \ \ \ d\eta(X^V, Y^H)=d\eta(X^H, Y^V)=0.
\]
Thus we deduce that $(\phi, \eta, \xi, G)$ is a paracontact metric Finsler structure if and only if the following hold
\begin{eqnarray*}
d\eta(X^H, Y^H)\!\!\!\!&=&\!\!\!\!d\eta^H(X, Y)=G^H(X, \phi Y),\\
d\eta(X^V, Y^V)\!\!\!\!&=&\!\!\!\!d\eta^V(X, Y)=G^V(X, \phi Y),\\
d\eta(X^H, Y^V)\!\!\!\!&=&\!\!\!\!d\eta(X^V, Y^H)=0.
\end{eqnarray*}
Moreover if this structure is normal then it is called \textit{para-Sasakian Finsler structure}.

Let $(\phi, \eta, \xi, G)$ be a paracontact metric Finsler structure on $E$. If  $\xi^H$ and $\xi^V$ are Killing  vector fields with respect to $G^H$ and $G^V$, respectively, then $(\phi, \eta, \xi, G)$  is called a \textit{$K$-paracontact Finsler structure} on $E$ and $E$ is called a \textit{$K$-paracontact Finsler vector bundle}.
\begin{theorem}\label{hasan1}
Let $ (\phi, \eta, \xi, G) $ be a paracontact metric Finsler structure on $E$. Then $ N^{(2)}=N^{(4)}=0 $. Moreover $ N^{(3)}=0 $ if and only if $\xi^H$ and $\xi^V$ are Killing vector fields with respect to $G^H$ and $G^V$, respectively.
\end{theorem}
\begin{proof}
Since $ (\phi, \eta, \xi, G) $ is a paracontact metric Finsler structure on $E$, then we have
\begin{align*}
0&=G^H(\xi^H,\phi X^H)=d\eta^H(\xi^H, X^H)=(\pounds_{\xi}^H\eta^H)(X^H)=N^{(4)}(X^H).
\end{align*}
We have also
\[
d\eta^H(\phi X^H, Y^H)=G^H(\phi X^H,\phi Y^H)=-G^H(X^H,\phi^2 Y^H)=-d\eta^H(X^H,\phi Y^H),
\]
which gives us $N^{(2)}(X^H, Y^H)=0$. Similarly, we obtain $N^{(2)}(X^V, Y^V)=0$. Using (\ref{com}) and (\ref{com2}),  we get
\[
d\eta^H(\phi X^V, Y^H)=d\eta^H(\phi Y^H, X^V)=d\eta^V(\phi X^V, Y^H)=d\eta^V(\phi Y^H, X^V)=0.
\]
The above equations give us $N^{(2)}(X^V, Y^H)=0$.

Now, we prove the second part of the Theorem. The following holds
\[
\pounds_{\xi}^H d \eta^H=i_\xi^H(d^2\eta^H)+d\circ i_\xi^Hd\eta^H=d\circ i_\xi^Hd\eta^H.
\]
Since $N^{(4)}=0$, then we obtain
\begin{equation}\label{110}
(i_{\xi^H}d\eta^H)(X^H)=d\eta^H(\xi^H, X^H)=N^{(4)}(X^H)=0.
\end{equation}
By assumption, we have
\be
d\eta^H(\xi^H, X^V)=G^H(\xi^H, \phi X^V)=0.\label{o4}
\ee
By (\ref{o4}), it follows that
\begin{equation}\label{111}
(i_{\xi^H}d\eta^H)(X^V)=d\eta^H(\xi^H, X^V)=0.
\end{equation}
Then (\ref{110}) and (\ref{111}) imply that $i_{\xi^H}d\eta^H=0$ and consequently $\pounds_{\xi}^H d \eta^H=0$. Similarly we obtain $\pounds_{\xi}^V d \eta^V=0$. Therefore  we get
\begin{align}
0&=(\pounds_{\xi}^H d\eta^H)(X, Y^H)=(\pounds_{\xi}^HG^H)(X,\phi Y^H)+G^H(X, (\pounds^H_\xi\phi)(Y^H)),\label{God1}\\
0&=(\pounds_{\xi}^H d\eta^H)(X, Y^V)=(\pounds_{\xi}^HG^H)(X,\phi Y^V)+G^H(X, (\pounds^H_\xi\phi)(Y^V)),\label{God2}\\
0&=(\pounds_{\xi}^V d\eta^V)(X, Y^H)=(\pounds_{\xi}^VG^V)(X,\phi Y^H)+G^V(X, (\pounds^V_\xi\phi)(Y^H)),\label{God3}\\
0&=(\pounds_{\xi}^V d\eta^V)(X, Y^V)=(\pounds_{\xi}^VG^V)(X,\phi Y^V)+G^V(X, (\pounds^V_\xi\phi)(Y^V)).\label{God4}
\end{align}
By  these equations,  we conclude that if $\pounds_{\xi}^HG^H=\pounds_{\xi}^VG^V=0$, then $N^{(3)}=0$.

Conversely, let $N^{(3)}=0$. Then from (\ref{God1})-(\ref{God4}) we get
\begin{equation}\label{God5}
(i)\ (\pounds_{\xi}^HG^H)(X,\phi Y)=0,\ \ \ (ii)\ (\pounds_{\xi}^VG^V)(X,\phi Y)=0.
\end{equation}
Now, we show that $(\pounds_{\xi}^HG^H)(X, Y)=0$. It is easy to see that
\[
(\pounds_{\xi}^HG^H)(X^V, Y^V)=0.
\]
Using part (i) of (\ref{God5}),  we obtain
\begin{align}\label{God6}
(\pounds_{\xi}^HG^H)(X^H, Y^H)&=(\pounds_{\xi}^HG^H)(X^H, \phi^2Y^H)+\eta^H(Y^H)(\pounds_{\xi}^HG^H)(X^H, \xi^H)\nonumber\\
&=\eta^H(Y^H)(\pounds_{\xi}^HG^H)(X^H, \xi^H).
\end{align}
Since $N^{(4)}=0$, then we have
\begin{equation}\label{God7}
(\pounds_{\xi}^HG^H)(X^H, \xi^H)=(\pounds_\xi^H\eta^H)(X^H)=0.
\end{equation}
The relations (\ref{God6}) and (\ref{God7}) give us
\[
(\pounds_{\xi}^HG^H)(X^H, Y^H)=0.
\]
By  part (i) of (\ref{God5}),  we get
\begin{align}\label{God8}
(\pounds_{\xi}^HG^H)(X^H, Y^V)&=(\pounds_{\xi}^HG^H)(X^H, \phi^2Y^V)+\eta^V(Y^V)(\pounds_{\xi}^HG^H)(X^H, \xi^V)\nonumber\\
&=-\eta^V(Y^V)G^H(X^H, [\xi^H, \xi^V]).
\end{align}
Again by using part (i) of (\ref{God5}), it follows that
\begin{align}\label{God9}
0&=(\pounds_{\xi}^HG^H)(\xi^V, \phi^2Y^H)=-G^H([\xi^H, \xi^V], \phi^2 Y^H)\nonumber\\
&=-G^H([\xi^H, \xi^V], Y^H)+\eta^H(Y^H)G^H([\xi^H, \xi^V], \xi^H)\nonumber\\
&=-G^H([\xi^H, \xi^V], Y^H)+\eta^H(Y^H)\eta^H([\xi^H,\xi^V]).
\end{align}
Since $N^{(4)}=0$, then we have
\begin{equation}\label{God10}
0=(\pounds_{\xi^V}\eta^H)(\xi^H)=-\eta^H([\xi^V, \xi^H]).
\end{equation}
Plugging (\ref{God10}) in (\ref{God9}) implies that
\begin{equation}\label{God11}
G^H([\xi^H, \xi^V], Y^H)=0.
\end{equation}
Then   (\ref{God8}) reduces to the following
\[
(\pounds_{\xi}^HG^H)(X^H, Y^V)=0.
\]
It results that $(\pounds_{\xi}^HG^H)(X, Y)=0$, where $X, Y\in\chi(E)$. Similarly, we can obtain $(\pounds_{\xi}^VG^V)(X, Y)=0$. This completes the proof.
\end{proof}

\bigskip
In the next proposition, we explain an important relation as a big gadget for our next  purposes.
%Using (\ref{l2}) we have the following relations for Finsler connection
%\begin{align}\label{l3}
%2G(D_X^HY^H, Z^H)&=X^HG(Y^H, Z^H)+Y^HG(X^H, Z^H)-Z^HG(X^H,Y^H)\nonumber\\
%&\ \ \ +G([X^H, Y^H], Z^H)-G([X^H, Z^H], Y^H)-G([Y^H, Z^H], X^H),
%\end{align}
%and
%\begin{align}\label{l4}
%2G(D_X^VY^V, Z^V)&=X^VG(Y^V, Z^V)+Y^VG(X^V, Z^V)-Z^VG(X^V,Y^V)\nonumber\\
%&\ \ \ +G([X^V, Y^V], Z^V)-G([X^V, Z^V], Y^V)-G([Y^V, Z^V], X^V).
%\end{align}
%\begin{lemma}\label{esi1}
%For any almost Paracontact Finsler structure $ (\phi,\eta,\xi,G) $ on $E$, we have
%\begin{align}\label{shari}
%&2G((D_{_{X}} \phi)Y,Z)+G(T(\phi Y, X), Z)+G(T(X, Z), \phi Y)+G(T(\phi Y, Z), X)\nonumber\\
%&+G(T(Y, X), \phi Z)+G(T(X, \phi Z), Y)+G(T(Y, \phi Z), X)=-d\Phi(X,\phi Y,\phi Z)\nonumber\\
%&-d\Phi(X,Y,Z)-G( N^{(1)}(Y,Z),\phi X)+ N^{(2)}(Y,Z)\eta(X)+d \eta(\phi Y,X)\eta(Z)\nonumber\\
%&-d \eta(\phi Z,X)\eta(Y),
%\end{align}
%where $D$ is a Finsler connection compatible with $G$.
%\end{lemma}

\smallskip

\begin{proposition}\label{esi1}
Let $(\phi,\eta,\xi,G)$ be an almost Paracontact metric Finsler structure  on $E$. Then the following hold
\begin{align}
2G((D_{_{X^H}} \phi)Y^H,Z^H)&=-d\Phi(X^H,\phi Y^H,\phi Z^H)-d\Phi(X^H,Y^H,Z^H)\nonumber\\
& -G\big( N^{(1)}(Y^H,Z^H),\phi X^H\big)+N^{(2)}(Y^H,Z^H)\eta(X^H)\nonumber\\
&+d\eta^H(\phi Y^H,X^H)\eta(Z^H)-d\eta^H(\phi Z^H,X^H)\eta(Y^H),\label{mohi}\\
2G((D_{_{X^V}} \phi)Y^V,Z^V)&=-d\Phi(X^V,\phi Y^V,\phi Z^V)-d\Phi(X^V,Y^V,Z^V)\nonumber\\
& -G( N^{(1)}(Y^V,Z^V),\phi X^V)+N^{(2)}(Y^V,Z^V)\eta(X^V)\nonumber\\
& +d\eta^V(\phi Y^V,X^V)\eta(Z^V)-d\eta^V(\phi Z^V,X^V)\eta(Y^V).\label{mohi1}
\end{align}
\end{proposition}
\begin{proof}
By a simple calculation, we get
\begin{eqnarray*}
d\Phi(X^H,\phi Y^H,\phi Z^H)=\!\!\!\!&-&\!\!\!\! X^H(\Phi(Y^H,Z^H))-\phi Y^H(g(Z^H,X^H))\\
\!\!\!\!&+&\!\!\!\! \phi Y^H(\eta(Z^H)\eta(X^H))+\phi Z^H(G(X^H,Y^H))\\
\!\!\!\!&-&\!\!\!\! \phi Z^H(\eta(X^H)\eta(Y^H))-G([X^H,\phi Y^H],Z^H)\\
\!\!\!\!&+&\!\!\!\! \eta^H([X^H,\phi Y^H])\eta(Z^H)-G([\phi Z^H,X^H],Y^H)\\
\!\!\!\!&+&\!\!\!\! \eta^H([\phi Z^H,X^H]\eta(Y^H)-\Phi([\phi Y^H,\phi Z^H],X^H).
\end{eqnarray*}
Also we have
\begin{align*}
G(N^{(1)}(Y^H,Z^H),\phi X^H)&=\Phi([Y^H,Z^H],X^H)+\Phi([\phi Y^H,\phi Z^H],X^H)\\
&\ \ +G([\phi Y^H,Z^H],X^H)-\eta^H([\phi Y^H,Z^H])\eta(X^H)\\
&\ \ +G([Y^H,\phi Z^H],X^H)-\eta^H([Y^H,\phi Z^H])\eta(X^H),
\end{align*}
Moreover the following hold
\begin{eqnarray*}
d\eta^H(\phi Y^H,X^H)\eta(Z^H)\!\!\!\!&=&\!\!\!\! \phi Y^H(\eta(X^H))\eta(Z^H)-\eta^H([\phi Y^H,X^H])\eta(Z^H),\\
d\eta^H(\phi Z^H,X^H)\eta(Y^H)\!\!\!\!&=&\!\!\!\!\phi Z^H(\eta(X^H))\eta(Y^H)-\eta^H([\phi Z^H,X^H])\eta(Y^H).
\end{eqnarray*}
If we denote righthand side of (\ref{mohi}) by $I$, then by using the above equations we can obtain the following
\begin{eqnarray}
\nonumber I=\!\!\!\!&&\!\!\!\!\phi Y^H(G(Z^H,X^H))-\phi Z^H(G(X^H,Y^H))+G([X^H,\phi Y^H],Z^H)\\
\nonumber \!\!\!\!&+&\!\!\!\!G([\phi Z^H,X^H],Y^H)-Y^H(\Phi(Z^H,X^H))-Z^H(\Phi(X^H,Y^H))\\
\nonumber \!\!\!\!&+&\!\!\!\!\Phi([X^H,Y^H],Z^H)+\Phi([Z^H,X^H],Y^H)-G([\phi Y^H,Z^H],X^H)\\
\!\!\!\!&-&\!\!\!\!G([Y^H,\phi Z^H],X^H).\label{o5}
\end{eqnarray}
Since $D$ is a Finsler connection, then it is $G$-compatible and it's $(h)h$-torsion vanishes. Thus (\ref{o5}) reduces to following
\begin{equation}\label{Fins2}
I=G((D_{X^H}\phi)Y^H, Z^H)-G(D_{X^H}Z^H, \phi Y^H)-G(D_{X^H}\phi Z^H, Y^H).
\end{equation}
On the other hand,  we have
\begin{align}
X^HG(Z^H, \phi Y^H)&=G(D_{X^H}Z^H, \phi Y^H)+G(Z^H, D_{X^H}\phi Y^H),\label{Fins}\\
X^HG(\phi Z^H, Y^H)&=G(D_{X^H}\phi Z^H, Y^H)+G(\phi Z^H, D_{X^H}Y^H).\label{Fins1}
\end{align}
Since $G(Z^H, \phi Y^H)=G(\phi Z^H, Y^H)$, then by (\ref{Fins}) and (\ref{Fins1}) we get
\be
G(D_{X^H}Z^H, \phi Y^H)+G(D_{X^H}\phi Z^H, Y^H)=-G((D_{X^H}\phi)Y^H, Z^H).\label{o6}
\ee
Plugging (\ref{o6}) in (\ref{Fins2}) give us (\ref{mohi}). Similarly we can obtain (\ref{mohi1}).
\end{proof}

\bigskip

\begin{proposition}\label{hasan}
Let $(\phi,\eta,\xi,G)$ be a paracontact metric Finsler structure on $E$. Then the following hold
\begin{align}
2G((D_{_{X^H}} \phi)Y^H,Z^H)&=-G( N^{(1)}(Y^H,Z^H),\phi X^H)+d\eta^H(\phi Y^H,X^H)\eta(Z^H)\nonumber\\
&\ \ \ -d\eta^H(\phi Z^H,X^H)\eta(Y^H),\label{shari1}\\
2G((D_{_{X^V}} \phi)Y^V,Z^V)&=-G( N^{(1)}(Y^V,Z^V),\phi X^V)+d\eta^V(\phi Y^V,X^V)\eta(Z^V)\nonumber\\
&\ \ \ -d\eta^V(\phi Z^V,X^V)\eta(Y^V).\label{shari11}
\end{align}
Moreover we get $D_{_\xi}\phi=0$.
\end{proposition}
\begin{proof}
By Proposition \ref{esi1}, we can get (\ref{shari1}), (\ref{shari11}). Thus we  prove $D_\xi\phi=0$ . By $N^{(2)}=0$,  we obtain $d\eta^H(\phi X^H, \xi^H)=0$. Thus by plugging  $X=\xi^H$ in (\ref{shari1}) we get the following
\[
G((D_{\xi^H}\phi)Y^H, Z^H)=0,
\]
which gives us
\[
G^H((D_{\xi^H}\phi)Y^H, Z)=0.
\]
We have also,  $G^H((D_{\xi^H}\phi)Y^V, Z)=0$. Therefore we obtain
\[
G^H((D_{\xi^H}\phi)Y, Z)=0.
\]
It means that $D_{\xi^H}\phi=0$. Similarly, we get $D_{\xi^V}\phi=0$. Therefore $D_{\xi}\phi=0$.
\end{proof}

\bigskip

Using Theorem \ref{hasan1},  we conclude the following.
\begin{theorem}\label{kpara}
Let $ (\phi,\eta,\xi,G) $ is a Paracontact metric Finsler structure on $E$. Then this structure is a K-paracontact structure if and only if $N^{(3)}=0 $.
\end{theorem}
Since a para-Sasakian Finsler structure is normal, then we have $N^{(3)}=0$. Thus from the above proposition we deduce the following
\begin{cor}\label{kpara1}
A para-Sasakian structure on $E$ is $K$-paracontact structure.
\end{cor}
\bigskip
Now, we are going to find some conditions under which a paracontact metric Finsler structure on  vector bundle $E$ reduces to a K-paracontact Finsler structure. More precisely, we prove the following theorem.

\begin{theorem}\label{p1}
Let $(\phi,\eta,\xi,G)$ be a paracontact metric Finsler structure on $ E$. Then this structure is a K-paracontact Finsler structure if and only if
\begin{equation}\label{feri}
\left\{
\begin{array}{cc}
(i)\ D^H_{X}\xi^H=-\dfrac{1}{2}\phi X^H, &(ii)\ G^H([\xi^H, X^V]^H, Y^H)=0,\\
(iii)\ D^V_{X}\xi^V=-\dfrac{1}{2}\phi X^V, &(iv)\ G^V([\xi^V, X^H]^V, Y^V)=0.
\end{array}
\right.
\end{equation}
\end{theorem}
\begin{proof}
Let $(\phi,\eta,\xi,G)$ be a K-paracontact Finsler structure. Then the following holds
\[
\pounds_\xi^HG^H=\pounds_\xi^VG^V=0.
\]
We have
\begin{align*}
0&=(\pounds_\xi^HG^H)(X^V, Y^H)=-G^H([\xi^H, X^V]^H, Y^H),\\
0&=(\pounds_\xi^VG^V)(X^H, Y^V)=-G^V([\xi^V, X^H]^V, Y^V),
\end{align*}
which give us (ii) and (iv) of (\ref{feri}).

It is easy to see that, the following holds
\[
(\pounds _{_\xi}^HG)(X^H,Y^H)=(\pounds _{_\xi}^HG^H)(X^H,Y^H).
\]
Therefore
\begin{align}\label{l}
0&=(\pounds _{_\xi}^HG)(X^H,Y^H)=\pounds _{_\xi}^HG(X^H,Y^H)-G(\pounds _{_\xi}^HX^H,Y^H)-G(X^H,\pounds _{_\xi}^HY^H)\nonumber\\
&=\pounds _{_\xi}^HG(X^H,Y^H)-G([\xi^H, X^H]^H,Y^H)-G(X^H,[\xi^H,Y^H]^H).
\end{align}
Since $D$ is symmetric, then we have
\begin{equation}\label{l1}
[\xi^H, X^H]^H=D^H_\xi X^H-D^H_X\xi^H.
\end{equation}
Plugging (\ref{l1}) in (\ref{l}) yields that
\[
0=(D _{_\xi}^HG)(X^H,Y^H)+G(D_{_X}^H \xi^H,Y^H)+G(X^H,D_{_Y}^H \xi^H).
\]
Since $D$ is $G$-compatible, then $D _{_\xi}^HG=0$. Thus
\begin{equation}\label{1}
G(D_{_X}^H \xi^H,Y^H)=-G(X^H,D_{_Y}^H \xi^H).
\end{equation}
Similarly we get
\begin{equation}\label{6}
G(D_{_X}^V \xi^V,Y^V)=-G(X^V,D_{_Y}^V \xi^V).
\end{equation}
Using (\ref{l2}), we obtain
\begin{equation}\label{7}
2G(D_{_X}^H\xi^H,Y^H)-2G(X^H,D_{_Y}^H \xi^H)=2d \eta(X^H,Y^H).
\end{equation}
By  (\ref{1}) and (\ref{7}) we have
\begin{equation}\label{8}
2G(D_{_X}^H\xi^H,Y^H)-2G(X^H,D_{_Y}^H \xi^H)=4G(D_{_X}^H\xi^H,Y^H).
\end{equation}
(\ref{7}) and (\ref{8}) give us
$$2G(D_{_X}^H\xi^H,Y^H)=d \eta(X^H,Y^H)=G(X^H,\phi Y^H)=-G(\phi X^H,Y^H).$$
Hence
\[
D^H_{X}\xi^H=-\dfrac{1}{2}\phi X^H.
\]
Similarly by using (\ref{6}),  we can deduce that $D^V_{X}\xi^V=-\dfrac{1}{2}\phi X^V $.

Conversely, suppose that  (\ref{feri}) holds. Then from part (i) of  (\ref{feri}) we have
\begin{align*}
0&=(\pounds _{_\xi}^HG^H)(X^H,Y^H)\\
&=G(D_{_X}^H \xi^H,Y^H)+G(X^H,D_{_Y}^H \xi^H)\\
&=-\frac{1}{2}[G(\phi X^H,Y^H)+G(X^H,\phi Y^H)]=0.
\end{align*}
Also (ii) gives us
\[
(\pounds_{\xi}^HG^H)(X^V,Y^H)=0.
\]
Therefore by considering
\[
(\pounds _{\xi}^HG^H)(X^V,Y^V)=0,
\]
we deduce $\pounds_{\xi}^HG^H=0$. By a similar method, we can obtain $\pounds_{\xi}^VG^V=0$. This completes the proof.
\end{proof}

\bigskip

\begin{lemma}\label{4.7}
Let  $(\phi,\eta,\xi,G)$ be a K-paracontact Finsler structure on a vector bundle $E$. Then the following hold
\begin{align}
R(X^V, \xi^V)\xi^V&=-\frac{1}{4}(X^V-\eta^V(X^V)\xi^V),\label{babi}\\
R(X^H, \xi^H)\xi^H&=-\frac{1}{4}(X^H-\eta^H(X^H)\xi^H)-D^V_{[X^H, \xi^H]}\xi^H. \label{babi1}
\end{align}
\end{lemma}
\begin{proof}
Using $[X^V, \xi^V]^H=0$,  $D_\xi\phi=0$ and  (\ref{Rie}), we obtain
\begin{eqnarray*}
R(X^V, \xi^V)\xi^V\!\!\!\!&=&\!\!\!\!\frac{1}{2}\phi(D_{\xi^V}X^V+[X^V, \xi^V]^V)\\
\!\!\!\!&=&\!\!\!\! -\frac{1}{4}\phi^2(X^V)\\
\!\!\!\!&=&\!\!\!\!-\frac{1}{4}\big[X^V-\eta^V(X^V)\xi^V\big].
\end{eqnarray*}
Similarly we have
\begin{align*}
R(X^H, \xi^H)\xi^H&=\frac{1}{2}\phi(D_{\xi^H}X^H+[X^H, \xi^H]^H)-D_{[X^H, \xi^H]^V}\xi^H\nonumber\\
&=-\frac{1}{4}\phi^2(X^H)-D_{[X^H, \xi^H]^V}\xi^H\nonumber\\
&=-\frac{1}{4}(X^H-\eta^H(X^H)\xi^H)-D^V_{[X^H, \xi^H]}\xi^H.
\end{align*}
This completes the proof.
\end{proof}

\bigskip

\begin{theorem}\label{chakhan}
Let $(\phi,\eta,\xi,G)$ be a  K-paracontact Finsler structure on $E$. Then the following hold
\begin{description}
  \item[(i)] the vertical flag curvature of all plane sections containing $\xi^V$ are equal to $-\frac{1}{4}$;
  \item[(ii)] the horizontal flag curvature of all plane sections containing $\xi^H$ are equal to $-\frac{1}{4}$ if and only if $G(D^V_{[X^H, \xi^H]}\xi^H, X^H)=0$.
\end{description}
\end{theorem}
\begin{proof}
Let $X^V$ be a unit vector field orthogonal to $\xi^V$. Then
\[
\eta^V(X^V)=0.
\]
Consequently (\ref{babi}) gives us
\[
R(X^V, \xi^V)\xi^V=-\frac{1}{4}X^V.
\]
Therefore we get
\[
K(X^V, \xi^V)=G^V(R(X^V, \xi^V)\xi^V, X^V)=-\frac{1}{4}G(X^V, X^V)=-\frac{1}{4}.
\]
Similarly, if $X^H$ is a unit vector field orthogonal to $\xi^H$, then from (\ref{babi1}) we get
\begin{align*}
K(X^H, \xi^H)&=G^H(R(X^H, \xi^H)\xi^H, X^H)\\
&=-\frac{1}{4}G(X^H, X^H)-G(D_{[X^H, \xi^H]^V}\xi^H, X^H)\\
&=-\frac{1}{4}-G(D_{[X^H, \xi^H]^V}\xi^H, X^H).
\end{align*}
Therefore  $K(X^H, \xi^H)=-\frac{1}{4}$ holds if and only if $G(D_{[X^H, \xi^H]^V}\xi^H, X^H)=0$.
\end{proof}

\bigskip

Now, we are going to study some properties of   para-Sasakian Finsler structure on a vector bundle. First, we prove the following.

\smallskip

\begin{theorem}
Let $(\phi, \eta, \xi,G)$ be a para-Sasakian Finsler structure  on a  vector bundle $E$. Then the following relations hold
\begin{equation}\label{2}
(D_{_X}^H\phi)Y^H=\dfrac{1}{2}\{\eta^H(Y^H)X^H-G^H(X^H,Y^H)\xi^H\},
\end{equation}
\begin{equation}\label{3}
(D_{_X}^V\phi)Y^V=\dfrac{1}{2}\{\eta^V(Y^V)X^V-G^V(X^V,Y^V)\xi^V\},
\end{equation}
Moreover, the Riemannian curvature satisfies the following
\begin{align}
R(X^V,Y^V)\xi^V&=\dfrac{1}{4}\{\eta^V(X^V)Y^V-\eta^V(Y^V)X^V\},\label{salar1}\\
R(X^H,Y^H)\xi^H&=\dfrac{1}{4}\{\eta^H(X^H)Y^H-\eta^H(Y^H)X^H\}-D^V_{[X^H, Y^H]}\xi^H\label{salar2}.
\end{align}
\end{theorem}
\begin{proof}
Since $(\phi, \eta, \xi,G)$ is para-Sasakian Finsler structure, then $ \Phi=d \eta $ and $ N^{(1)}=N^{(2)}=0 $. Thus by (\ref{mohi}),  we obtain
\begin{align*}
2G((D^H_X\phi)Y^H, Z^H)&=d\eta^H(\phi Y^H, X^H)\eta(Z^H)-d\eta^H(\phi Z^H, X^H)\eta(Y^H)\\
&=G(\phi Y^H, \phi X^H)\eta(Z^H)-G(\phi Z^H, \phi X^H)\eta(Y^H)\\
&=-G(X^H, Y^H)\eta(Z^H)+G(X^H, Z^H)G(\xi^H, Y^H)\\
&=G(\eta(Y^H)X^H-G(X^H, Y^H)\xi^H, Z^H).
\end{align*}
This implies (\ref{2}). With similar computations, one can obtain (\ref{3}).

Using (\ref{Rie}), Theorem \ref{kpara} and Corollary \ref{kpara1} we have
\begin{eqnarray}
\nonumber R(X^V,Y^V)\xi^V\!\!\!\!&=&\!\!\!\!D_X^VD_Y^V \xi^V-D_Y^VD_X^V\xi^V-D_{[X^V,Y^V]}^V\xi^V\\
\nonumber  \!\!\!\!&=&\!\!\!\!D_{_X}^V(-\dfrac{1}{2}\phi Y^V)-D_{_Y}^V(-\dfrac{1}{2}\phi X^V)+\dfrac{1}{2}\phi[X^V,Y^V]^V\\
\!\!\!\!&=&\!\!\!\!-\dfrac{1}{2}(D_X^V\phi)Y^V+\dfrac{1}{2}(D_Y^V\phi)X^V.\label{o7}
\end{eqnarray}
By  (\ref{3}) and   (\ref{o7}) we get
\[
R(X^V,Y^V)\xi^V=\dfrac{1}{4}\{\eta^V(X^V)Y^V-\eta^V(Y^V)X^V\}.
\]
Similarly by using (\ref{2}), we obtain
\begin{eqnarray*}
R(X^H,Y^H)\xi^H\!\!\!\!&=&\!\!\!\!D_X^HD_Y^H \xi^H-D_Y^HD_X^H\xi^H-D_{[X^H,Y^H]}^H\xi^H-D_{[X^H,Y^H]}^V\xi^H\\
\!\!\!\!&=&\!\!\!\!D_{_X}^H(-\dfrac{1}{2}\phi Y^H)-D_{_Y}^H(-\dfrac{1}{2}\phi X^H)+\dfrac{1}{2}\phi[X^H,Y^H]^H-D_{[X^H,Y^H]}^V\xi^H\\
\!\!\!\!&=&\!\!\!\!-\dfrac{1}{2}(D_X^H\phi)Y^H+\dfrac{1}{2}(D_Y^H\phi)X^H-D_{[X^H,Y^H]}^V\xi^H\\
\!\!\!\!&=&\!\!\!\! \dfrac{1}{4}\{\eta^H(X^H)Y^H-\eta^H(Y^H)X^H\}-D_{[X^H,Y^H]}^V\xi^H.
\end{eqnarray*}
This completes the proof.
\end{proof}

\bigskip

A plane section in $V_uE$ is called a vertical $\phi$-section if there exist a unit vector $X^V$ in $V_uE$ orthogonal to $\xi^V$ such that $\{X^V, \phi X^V\}$ span the section. The vertical flag curvature $K(X^V, \phi X^V)$ is called the \textit{vertical $\phi$-flag curvature.}

\smallskip

\begin{proposition}
Let $ (\phi, \eta, \xi,G) $ be a para-Sasakian Finsler structure on $E$. Suppose that $E$ is locally symmetric. Then it has vertical $\phi$-flag curvature $-\frac{1}{4}$.
\end{proposition}
\begin{proof}
Let $X^V\neq 0$ be a vector field on $E$ orthogonal to $\xi^V$. Then we have $\eta^V(X^V)=G^V(X^V,\eta^V)=0$. By direct conclusion we obtain
\begin{eqnarray}\label{Hi}
\nonumber(D_{\phi X^V}R)(X^V, \phi X^V)\xi^V=\frac{1}{2}\!\!\!\!&\Big[&\!\!\!\! \phi R(X^V, \phi X^V)\phi X^V- \frac{1}{4}G^V(X^V, \phi X^V)\phi^2X^V\\
\!\!\!\!&+&\!\!\!\!\frac{1}{4}G^V(\phi X^V, \phi X^V)\phi X^V\Big].
\end{eqnarray}
Considering $G(X^V, \phi X^V)=-G(\phi X^V, X^V)$, we have $G(X^V, \phi X^V)=0$. Using this equation and noting that $E$ is locally symmetric, (\ref{Hi}) gives us
\be\label{o8}
\phi R(X^V, \phi X^V)\phi X^V+\frac{1}{4}G^V(\phi X^V, \phi X^V)\phi X^V=0.
\ee
By (\ref{o8}),  we get
\be\label{o9}
G(\phi R(X^V, \phi X^V)\phi X^V, \phi X^V)+\frac{1}{4}G(\phi X^V, \phi X^V)G(\phi X^V, \phi X^V)=0.
\ee
Since $\eta^V(X^V)=0$, then (\ref{o9}) gives us
\[
G(R(X^V, \phi X^V)\phi X^V, X^V)=\frac{1}{4}G^2(\phi X^V, \phi X^V).
\]
Therefore we obtain
\[
K(X^V, \phi X^V)=\frac{G(R(X^V, \phi X^V)\phi X^V, X^V)}{G(X^V, X^V)G(\phi X^V, \phi X^V)}=-\frac{1}{4}.
\]
It means that $E$ has vertical $\phi$-flag curvature $-\frac{1}{4}$.
\end{proof}

\subsection{Horizontal and Vertical Ricci Tensors}

The \textit{horizontal Ricci tensor} $S^H$ of an $(2(k_1+k_2)+2)$-dimensional para-Sasakian Finsler manifold $E$ is given by
\begin{align}
S^H(X^H, Y^H)&=\sum_{i=1}^{2k_1}G(R(X^H, E_i^H)E_i^H, Y^H)+G(R(X^H, \xi^H)\xi^H, Y^H)\nonumber\\
&=\sum_{i=1}^{2k_1}G(R(E_i^H, X^H)Y^H, E_i^H)+G(R(\xi^H, X^H)Y^H, \xi^H),\label{salar7}
\end{align}
where $\{E_1^H, E_2^H,\ldots, E_{2k_1}^H, \xi^H\}$ is a local orthonormal frame of $H_uE$. Similarly, the \textit{vertical Ricci tensor} of an $(2(k_1+k_2)+2)$-dimensional para-Sasakian Finsler manifold $E$ is given by
\begin{align}
S^V(X^V, Y^V)&=\sum_{i=1}^{2k_2}G(R(X^V, E_i^V)E_i^V, Y^V)+G(R(X^V, \xi^V)\xi^V, Y^V)\nonumber\\
&=\sum_{i=1}^{2k_2}G(R(E_i^V, X^V)Y^V, E_i^V)+G(R(\xi^V, X^V)Y^V, \xi^V),\label{salar8}
\end{align}
where $\{E_1^V, E_2^V,\ldots, E_{2k_2}^V, \xi^V\}$ is a local orthonormal frame of $V_uE$.

\begin{proposition}
The  horizontal and vertical Ricci tensors $S^H$ and $S^V$ of a $(2(k_1+k_2)+2)$-dimensional para-Sasakian Finsler manifold satisfies the following
equations:
\begin{equation}\label{salar3}
\left\{
\begin{array}{cc}
(i)\ S^H(X^H, \xi^H)=-\frac{k_1}{2}\eta^H(X^H)-G(D^V_{[E_i^H, X^H]}\xi^H, E_i^H),&\\
\hspace{-3.7cm}(ii)\ S^V(X^V, \xi^V)=-\frac{k_2}{2}\eta^V(X^V),&\\
\hspace{-1.5cm}(iii)\ S^H(\xi^H, \xi^H)=-\frac{k_1}{2}-G(D^V_{[E_i^H, \xi^H]}\xi^H, E_i^H),&\\
\hspace{-4.8cm}(iv)\ S^V(\xi^V, \xi^V)=-\frac{k_2}{2}.
\end{array}
\right.
\end{equation}
\end{proposition}
\begin{proof}
Using (\ref{salar2}) and (\ref{salar7}), one can obtain the following:
\begin{eqnarray}\label{o10}
\nonumber S^H(X^H\!\!\!\!\!\!\!\!\!\!&,&\!\!\!\!\!\!\!\!\ \ \xi^H)= \sum_{i=1}^{2k_1}G\big(R(E_i^H, X^H)\xi^H,\ \ E_i^H\big)\\
\!\!\!\!&=&\!\!\!\! \sum_{i=1}^{2k_1}G\Big(\frac{1}{4}\eta^H(E_i^H)X^H-\frac{1}{4}\eta^H(X^H)E_i^H-D^V_{[E_i^H, X^H]}\xi^H,\  E_i^H\Big).
\end{eqnarray}
Since $E_i^H$ is orthogonal to $\xi^H$, then we have $\eta^H(E_i^H)=G(E_i^H, \xi^H)=0$. By (\ref{o10}) and  $G(E_i^H, E_i^H)=1$, we get the part (i) of (\ref{salar3}). Plugging $X^H=\xi^H$ in (i) and using $\eta^H(X^H)=1$ implies  (iii). Similarly, (\ref{salar1}) and (\ref{salar8}) give us
\begin{align}
\nonumber S^V(X^V, \xi^V)&=\sum_{i=1}^{2k_2}G\big(R(E_i^V, X^V)\xi^V,\ E_i^V\big)\\
\nonumber &=\frac{1}{4}\sum_{i=1}^{2k_2}G\Big(\eta^V(E_i^V)X^V-\eta^V(X^V)E_i^V,\ E_i^V\Big)\\
&=-\frac{1}{2}k_2\eta^V(X^V).\label{o11}
\end{align}
By setting $X^V=\xi^V$ in (\ref{o11}), we get (iv).
\end{proof}

\bigskip

According to the parts (i) and (iii) of (\ref{salar3}), one can deduces the following easily.
\begin{cor}
For a $(2(k_1+k_2)+2)$-dimensional para-Sasakian Finsler manifold, the following hold
\begin{description}
  \item[i)] $S^H(X^H, \xi^H)=-\frac{k_1}{2}\eta^H(X^H)$ is equivalent to vanishing of $G(D^V_{[E_i^H, X^H]}\xi^H, E_i^H)$;
  \item[ii)] $S^H(\xi^H, \xi^H)=-\frac{k_1}{2}$ is equivalent to vanishing of $G(D^V_{[E_i^H, \xi^H]}\xi^H, E_i^H)$.
\end{description}
\end{cor}

\bigskip

Using Lemma \ref{4.7}, we have the following proposition.
\begin{proposition}\label{propend}
The  horizontal and vertical Ricci tensors $S^H$ and $S^V$ of a $(2(k_1+k_2)+2)$-dimensional K-paracontact Finsler vector bundle satisfies the following equations:
\[
S^H(\xi^H, \xi^H)=-\frac{1}{2}k_1-G(D^V_{[E_i^H, \xi^H]}\xi^H, E_i^H),\ \ \ S^V(\xi^V, \xi^V)=-\frac{1}{2}k_2.
\]
\end{proposition}

\smallskip

Proposition \ref{propend},   have an easy  consequence as follows.
\begin{cor}
For a $(2(k_1+k_2)+2)$-dimensional K-paracontact Finsler vector bundle $E$, $S^H(\xi^H, \xi^H)=-\frac{k_1}{2}$ is equivalent to vanishing of $G(D^V_{[E_i^H, \xi^H]}\xi^H, E_i^H)$.
\end{cor}
%----------------------------------------------------------------------------------------

\bigskip

\noindent
Esmaeil Peyghan and Esa Sharahi\\
Department of Mathematics, Faculty  of Science\\
Arak University\\
Arak 38156-8-8349,  Iran\\
Email: epeyghan@gmail.com
\bigskip

\noindent
Akbar Tayebi\\
Department of Mathematics, Faculty  of Science\\
University of Qom \\
Qom. Iran\\
Email:\ akbar.tayebi@gmail.com

\end{document}